\documentclass[12pt]{amsart}
\usepackage{amssymb}
\title{Canonical filtrations on Harish-Chandra modules}
\author{Ivan Losev}
\newcommand{\g}{\mathfrak{g}}
\newcommand{\C}{\mathbb{C}}
\newcommand{\gr}{\operatorname{gr}}
\newcommand{\U}{\mathcal{U}}
\newcommand{\kf}{\mathfrak{k}}
\newcommand{\V}{\operatorname{V}}
\newcommand{\Nilp}{\mathcal{N}}

\newcommand{\Orb}{\mathbb{O}}
\newcommand{\HC}{\operatorname{HC}}
\newcommand{\Walg}{\mathcal{W}}
\newcommand{\Weyl}{\mathsf{A}}
\newcommand{\wHC}{\operatorname{wHC}}
\newcommand{\A}{\mathcal{A}}
\newcommand{\Supp}{\operatorname{Supp}}
\newcommand{\B}{\mathcal{B}}
\newcommand{\Hom}{\operatorname{Hom}}
\newcommand{\J}{\mathcal{J}}
\newcommand{\slf}{\mathfrak{sl}}
\newcommand{\Z}{\mathbb{Z}}
\newtheorem{Thm}{Theorem}[section]
\newtheorem{Prop}[Thm]{Proposition}

\newtheorem{Lem}[Thm]{Lemma}
\theoremstyle{definition}
\newtheorem{Ex}[Thm]{Example}
\newtheorem{defi}[Thm]{Definition}
\newtheorem{Rem}[Thm]{Remark}

\numberwithin{equation}{section}
\begin{document}
\begin{abstract}
The goal of this paper is to show that a wide class of  Harish-Chandra $(\g,K)$-modules including all irreducible ones come with 
a certain canonical filtration.
\end{abstract}
\maketitle
\section{Introduction}
Let $G$ be a semisimple algebraic group over $\C$ and $\g$ be its Lie algebra.
Let $\U$ denote the universal enveloping algebra $U(\g)$.

Fix an involution $\sigma$ of $\g$ and let $\kf$ be the fixed point locus $\g^\sigma$.
Let $K$ denote the corresponding connected algebraic subgroup of $G$ (one can deal with a more general choice of $K$
but for the purposes of the introduction we are not going to do this).
Then we can talk about Harish-Chandra (shortly, HC) $(\g,K)$-modules, i.e., finitely generated $\U$-modules, where the $\kf$-action integrates to $K$. Below we will only consider those HC modules
that have finite length. The category of such modules will be denoted by $\HC(\U,K)$.

Let $M$ be  HC $(\g,K)$-module. An important tool to study $M$ are {\it good filtrations} on $M$, i.e.,
$K$-stable $\U$-module filtrations such that $\gr M$ is finitely generated over $S(\g)$
(in fact, $\gr M$ is then an $S(\g/\kf)$-module). In general, there is no canonical
choice of such a filtration. The goal of this paper is to produce one under certain
additional condition on $M$ that, in particular, is satisfied by all irreducible modules.

In order to state the definition, we need to recall the notion of the associated variety of
a HC module. The support of $\gr M$ in $(\g/\kf)^*$ is independent of the choice of a good filtration
on $M$. It is called the {\it associated variety} of $M$ and is denoted by $\V(M)$.
Since we assume that $M$ has finite length, the associated variety $\V(M)$ is contained
in the nilpotent cone $\Nilp_K\subset (\g/\kf)^*$ and so is the union of finitely
many $K$-orbits. If $M$ is irreducible, then $\overline{G \V(M)}$ is the closure of a single
nilpotent orbit, it is the set of zeroes of the ideal $\gr \operatorname{Ann}_{\U}(M)$. For a $K$-stable closed subvariety $Y\subset \Nilp_K$, we write
$\HC_Y(\U,K)$ for the full subcategory of $\HC(\U)$ consisting of all $M$
such that $\V(M)\subset Y$. If $Y=X\cap (\g/\kf)^*$, where $X$ is a closed $G$-stable subset
in the nilpotent cone of $\g^*$, then we write $\HC_X(\U,K)$ for $\HC_Y(\U,K)$.

Now fix a nilpotent orbit $\Orb\subset \g^*$. Consider the full (but not abelian) subcategory
$\HC^\partial_\Orb(\U,K)\subset \HC_{\overline{\Orb}}(\U,K)$ consisting of all objects that do not
have nonzero submodules with associated variety contained in $\overline{\Orb}\setminus \Orb$.
For example, every simple HC module $M$ belongs to exactly one category
$\HC^\partial_{\Orb}(\U,K)$, here $\overline{\Orb}=\overline{G\V(M)}$.

Our goal is to equip $M$ with a canonical good filtration. The characterization only involves
the $S(\g)$-module $\gr M$ and has three conditions. Here are the first two conditions:
\begin{itemize}
\item[(a)] The annihilator $I$ of $\gr M$ in $S(\g)$ is radical.
\item[(b)] $\gr M$ has {\it depth at least 1} meaning that every element in $S(\g)/I$ that is not
a zero divisor in the algebra is not a zero divisor in $\gr M$ either.
\end{itemize}
The third condition involves a certain normalization. Take a $K$-equivariant graded
finitely generated $S(\g/\kf)$-module $N$ (an example
is provided by $\gr M$). Take a point $\chi$ in an open $K$-orbit
in $\V(M)$. Then the fiber $N_\chi$ is a finite dimensional module over the stabilizer
$K_\chi$ of $\chi$ in $K$. It turns out it comes also with a preferred choice of a grading.  Namely,
pick a nondegenerate invariant symmetric form $(\cdot,\cdot)$ on $\g$ and use it to identify
$\g$ with $\g^*$. Let $e\in \g^*$ be the element of $\g$ corresponding to $\chi$.
We include it into an $\slf_2$-triple $(e,h,f)$ with $h\in \kf$ and $f\in \kf^\perp$.
We note that $h$ is uniquely determined up to $K_\chi$-conjugacy. Let $\gamma:\C^\times\rightarrow G$
be the one-parameter subgroup with $d_1\gamma(1)=h$. In particular, the image lies in $K$.
Then the $\C^\times$-action on $\g^*$
given by $t.\alpha=t^{-2}\gamma(t)\alpha$ fixes $\chi$ and so gives a $\C^\times$-action (equivalently,
a grading) on $N_\chi$ to be denoted by $N_\chi=\bigoplus_{i} N_{\chi,i}$.

We impose the following normalization condition:

\begin{itemize}
\item[(c)] We have $(\gr M)_{\chi,i}=\{0\}$ for $i\neq 0,1$.
\end{itemize}

We note that this condition is independent of the choice of $h$ (different choices are conjugate by $K_\chi$ and so the resulting gradings are conjugate) and of the choice of $\chi$ up to the $K$-conjugacy.
So, we have one condition (c) for each $K$-orbit in $\V(M)\cap \Orb$.

Here is the main result of this paper.

\begin{Thm}\label{Thm:main}
Let $M\in \HC^\partial_\Orb(\U,K)$. There is a unique good filtration on $M$ (to be called canonical) such that (a) and (b) satisfied, while (c) is satisfied for all $K$-orbits in $\V(M)\cap \Orb$.
\end{Thm}

We will also see that the canonical filtrations are functorial.

\begin{Ex}\label{Ex:sl2}
Let $\g=\mathfrak{sl}_2$ and $\kf$ be its Cartan subalgebra. Pick numbers $z\in \C, a\in \Z$, and consider the HC module $M:=\U/\U(C-z,h-a)$, where $C$ is the Casimir element. It comes with the filtration induced from the usual filtration on $\U$. We claim that this filtration is canonical. The associated graded is
$\gr M=\C[e,f]/(ef)$, it satisfies conditions (a) and (b).
Now we check (c). There are two $K=\C^\times$-orbits in $\Orb\cap (\g/\kf)^*$: they are $(e\neq 0, f=0)$ and $(e=0, f\neq 0)$. The grading on $\gr M$ coming from the Kazhdan action associated to the former
has $\deg e=0, \deg f=4$. So the fiber of $(\gr M)$ at $e=1,f=0$ is a one-dimensional space in degree $0$.
The same conclusion holds for the other $K$-orbit. This establishes (c) for all orbits. So this filtration is canonical.
\end{Ex}

Using the canonical filtration (or any good filtration satisfying (b)) we will prove the following result conjectured by Vogan. Let $\V(M)^{\geqslant 2}$ denote the complement in $\V(M)$
to the union of $K$-orbits of dimension at most $\dim \V(M)-2$.

\begin{Thm}\label{Thm:codim1_connected}
Let $M$ be an irreducible HC $(\g,\kf)$-module, i.e., an irreducible $\g$-module, where $\kf$
acts locally finitely. Then the variety $\V(M)^{\geqslant 2}$ is connected.
\end{Thm}

The paper is organized as follows. The construction of the canonical filtration is crucially
based on the restriction functors for HC modules sketched in \cite[Section 6.1]{Wdim}
and elaborated in \cite[Section 5]{LY}. These functors will be recalled in
Section \ref{S_res_fun}. Section \ref{S_canon_filtr} is the main part of the paper, there we
construct canonical filtrations and establish their properties.

{\bf Acknowledgements}: I am grateful to Pavel Etingof and Victor Ginzburg for stimulating discussions.
My work has been partially supported by the NSF under grant DMS-2001139.

\section{Restriction functors}\label{S_res_fun}
In this section we will briefly recall the construction of the restriction functor
for HC modules following \cite[Section 6.1]{Wdim}. 
This functor goes from a suitable category of HC modules
to the category of modules over the W-algebra $\Walg$.
So, we first recall the W-algebras, Section \ref{SS_W_alg},
and then discuss a general setting for HC modules, Section \ref{SS_wHC}.
After that we recall the restriction functor and its properties, Section \ref{SS_res_fun_constr}.

\subsection{W-algebras}\label{SS_W_alg}
Let $\sigma$ be an involution of $\g$. Set $\kf:=\g^\sigma$. 
Take a nilpotent $G$-orbit $\Orb\subset \g^*$ whose intersection with $(\g/\kf)^*$ is non-empty 
and a $(G^\sigma)^\circ$-orbit $\Orb_K\subset \Orb\cap (\g/\kf)^*$.
Pick $\chi\in \Orb_K$ and consider the transverse Slodowy slice $S\subset \g^*$ through $\chi$.
We can and will assume that $S$ is stabilized by $\theta:=-\sigma:\g^*\rightarrow \g^*$.
The slice $S$ is stable with respect to the $\C^\times$-action recalled before Theorem
\ref{Thm:main}. The grading on $\C[S]$ induced by this action is positive. Also it is easy to
see that
\begin{equation}\label{eq:Slodowy_vanish1}
\C[S]_1=\{0\}.
\end{equation}

There is a natural Poisson bracket of degree $-2$ on $\C[S]$.
So $\C[S]$ is a graded Poisson algebra. It admits a filtered quantization $\Walg$.
This algebra was introduced by Premet in \cite{Premet}.
The definition we are using originates from \cite{Walg}, a more technical version is
given in  \cite[Section 2.2]{W_prim}. Let
$\U_\hbar$ denote the Rees algebra of $\U$ for the PBW filtration. We rescale the grading by $2$ so that
$\g$ and $\hbar$ are now in degree $2$. The anti-involution $\theta=-\sigma$ of $\g$
extends to a $\C[\hbar]$-linear anti-involution of $\U_\hbar$ again denoted by $\theta$. We can view $\chi$ as a homomorphism
$\U_\hbar\rightarrow \C$ and consider the completion $\U_\hbar^{\wedge_\chi}$
at the kernel of $\chi$. A standard argument shows that it is a flat $\U_\hbar$-module.
It inherits the anti-involution $\theta$ from $\U_\hbar$.

Set $V:=T_\chi\Orb$, this is a symplectic vector space with form $\omega$.
It is $\C^\times$-stable and also
stable under the anti-involution $\theta$ of $\g$. Consider its Weyl algebra $\Weyl$
and its homogenized version $\Weyl_\hbar=T(V)[\hbar]/([u,v]-\hbar \omega(u,v))$. Thanks to the decomposition $\g^*=T_\chi\Orb\oplus T_\chi S$,
we can view $V$ as a subspace in $(T_\chi \g^*)^*$. Then we can $\C^\times$- and $\theta$-equivariantly
lift this embedding to an embedding $V\hookrightarrow \U_\hbar^{\wedge_\chi}$ subject to the following condition:
\begin{itemize}
\item it lifts to an algebra homomorphism $\Weyl_\hbar^{\wedge_0}\hookrightarrow \U_\hbar^{\wedge_\chi}$, where $\bullet^{\wedge_0}$ stands for the completion at the maximal ideal of $0$.
\end{itemize}
Let $\Walg_\hbar'$ denote the centralizer of the image of $V$
in $\U_\hbar^{\wedge_\chi}$. Note that we have a decomposition
\begin{equation}\label{eq:algebra_decomposition}
\U_\hbar^{\wedge_\chi}=\Weyl_\hbar^{\wedge_0}\widehat{\otimes}_{\C[[\hbar]]}\Walg_\hbar',
\end{equation}
where $\widehat{\otimes}$ is the completed tensor product.

The algebra $\Walg_\hbar'$ comes with a $\C^\times$-action.
Let $\Walg_\hbar$ denote the $\C^\times$-finite part. Finally, set $\Walg:=\Walg_\hbar/(\hbar-1)\Walg_\hbar$. The algebra $\Walg$ is called the
{\it finite W-algebra}, it quantizes $\C[S]$ in the following sense: the grading on $\Walg_\hbar$
induces an ascending filtration $\Walg=\bigcup_{i\geqslant 0}\Walg_{\leqslant i}$, and we have a graded Poisson algebra isomorphism $\gr \Walg\xrightarrow{\sim}\C[S]$. We note that by the construction,
$\Walg$ comes with a {\it parity involution}, to be denoted by $\varsigma$, that preserves the filtration
and acts on $\C[S]=\gr\Walg$ as $-1\in \C^\times$. This is because $\hbar$ has degree $2$
for the $\C^\times$-action.

Also note that the algebra $\Walg$ (as a filtered algebra with a parity involution) only depends on $\Orb$ up to an isomorphism preserving these structures (the algebras for points $\chi,g\chi$ are isomorphic
via the action of $g$).


Also, $\Walg$ inherits the anti-involution $\theta$, it preserves the filtration and commutes with
$\varsigma$. This structure depends on $\Orb_K$ and not just on $\Orb$.

\subsection{Weakly HC modules}\label{SS_wHC}
Now we introduce a category of weakly HC $\U_\hbar$-modules from \cite[Section 2]{LY}.
More generally, let $\A_\hbar$ be a $\C[\hbar]$-algebra that is equipped with a
$\C[\hbar]$-linear anti-involution $\theta$. We assume that $A:=\A_\hbar/(\hbar)$
is commutative. Let $A^{-\theta}$ denote the $-1$-eigenspace for $\theta$. Consider the two-sided ideal
$\mathcal{J}_\hbar\subset \A_\hbar$ defined as the preimage of $AA^{-\theta}$.
This is the Lie subalgebra in $\U_\hbar$ with respect to the bracket given by $[a,b]_\hbar:=\hbar^{-1}(ab-ba)$.

Following \cite[Section 2.3]{LY}, by a {\it weakly HC} $\A_\hbar$-module we mean a finitely generated $\A_\hbar$-module $M_\hbar$
equipped, in addition, with a Lie algebra action of $\mathcal{J}_\hbar$, $(a,m)\mapsto a.m$,
subject to the following conditions:
\begin{enumerate}
		\item $am=\hbar(a.m)$, $\forall\, a \in \mathcal{J}_\hbar$.
		\item $(ba).m = b(a.m)$, $\forall\, a \in \mathcal{J}_\hbar, b \in \A_\hbar$.
		\item $a.(bm)=[a,b]_\hbar m+b(a.m), \forall a\in \mathcal{J}_\hbar, b\in \A_\hbar$.
\end{enumerate}

By a homomorphism of  weakly HC modules we mean an $\A_\hbar$-linear map
that intertwines the actions of $\mathcal{J}_\hbar$. The resulting category will be
denoted by $\wHC(\A_\hbar)$.

If $\A_\hbar$ is graded in such a way that $\deg \hbar=d$ and $\theta$ preserves the grading,
it makes sense to speak about graded weakly HC modules: the action of $\mathcal{J}_\hbar$
is of degree $-d$. The category of graded weakly HC modules will be denoted by $\wHC^{gr}(\A_\hbar)$.

Here is the most important special case. The anti-involution $\theta$ of $U(\g)$ lifts to
a graded $\C[\hbar]$-linear involution of $\U_\hbar$. Then $A^{-\theta}=S(\g)\kf$ and $\mathcal{J}_\hbar=\U_\hbar\kf+\hbar \U_\hbar$.
Let $M$ be a HC $(\g,\kf)$-module, meaning a finitely generated $U(\g)$-module with locally finite action of
$\kf$. Equip it with a good $\kf$-stable filtration.
Then its Rees module $R_\hbar(M)$ is weakly HC:
note that for every $a\in \mathcal{J}_\hbar$ and every $m\in R_\hbar(M)$ we have
$am\in \hbar M_\hbar$ and we define $a.m$ by $\hbar^{-1}(am)$: since $R_\hbar(M)$
is flat over $\C[\hbar]$, the division makes sense. Conversely, if $M_\hbar$ is a graded weakly HC
module over $\U_\hbar$, then $M_\hbar/(\hbar-1)M_\hbar$ is a HC $(\g,\kf)$-module. The grading on
$M_\hbar$ induces a good filtration on $M_\hbar/(\hbar-1)M_\hbar$.

We can apply this construction to the Lie algebra $\g\times \g^{opp}$ and the anti-involution
given by $(x,y)\mapsto (y,x)$. The resulting weakly HC modules will be referred to as weakly
HC $\U_\hbar$-bimodules. Note that they form a monoidal category, while $\wHC^{gr}(\U_\hbar)$
is its module category.

Another special case of $\A_\hbar$ is $\Walg_\hbar$ with its anti-involution $\theta$.

\subsection{Construction and properties}\label{SS_res_fun_constr}
Let $K$ denote a connected algebraic group with Lie algebra $\kf$ and a homomorphism 
to $G$ that induces the identity automorphism of $\kf$. Fix a character $\kappa$ of $K$
Let $\wHC^{gr}(\U_\hbar)^{K,\kappa}$ denote the category of {\it strongly $K$-equivariant} objects in
$\wHC^{gr}(\U_\hbar)$. By definition, this means that, for every $\xi\in \kf$, the operator 
$m\mapsto \xi m-\langle\kappa,\xi\rangle m$ coincides with the image of $\xi$ under the differential of
the $K$-action. Similarly, we can talk about HC $(\g,K,\kappa)$-modules, their category 
will be denoted by $\HC(\g,K,\kappa)$. Note that in Introduction 
we considered the situation when $K\hookrightarrow G$ and $\kappa=0$.

Now we proceed to the  {\it restriction functor} $\wHC^{gr}(\U_\hbar)^{K,\kappa}\rightarrow
\wHC^{gr}(\Walg_\hbar)$ to be denoted by $\bullet_{\dagger,\chi}$ (first sketched in \cite[Section 6.1]{Wdim}, see also \cite[Section 5.4]{LY}). It is constructed as
follows, compare to \cite[Section 5.4]{LY}.

Take $M_\hbar\in \wHC^{gr}(\U_\hbar)^{K,\kappa}$. Consider
its completion
$$M_\hbar^{\wedge_\chi}\cong \U_\hbar^{\wedge_\chi}\otimes_{\U_\hbar}R_\hbar(M).$$
We view $M_\hbar^{\wedge_\chi}$ as a module over the right hand side of (\ref{eq:algebra_decomposition}).

One can equip $M_\hbar^{\wedge_\chi}$ with a $\C^\times$-action as follows. Let $\iota$
denote the homomorphism $K\rightarrow G$. Set $d:=2|\ker\iota|$.
Then $dh/2\in \kf$ is the differential at $1$ of a one-parameter subgroup $\C^\times\rightarrow K$ to be denoted by $\gamma$ (one can take $d=2$ if $K\hookrightarrow G$).
Then we can define an action of $\C^\times$ on $M_\hbar$ by $t.m:=t^{di}\gamma(t)^{-1}m$
for $m\in M_{\hbar}$ of degree $i$. Note that the corresponding action of $\C^\times$ on $\g^*$
preserves $\chi$, and so extends to $M_\hbar^{\wedge_\chi}$. The latter becomes a $\C^\times$-equivariant
$\U_\hbar^{\wedge_\chi}$-module, where the $\C^\times$-action on $\U_\hbar^{\wedge_\chi}$
is rescaled $d/2$ times (so that $\hbar$ is in degree $d$).

One can show, \cite[Lemma 5.2]{LY}, that $M_\hbar^{\wedge_\chi}$ decomposes as $\C[[L,\hbar]]\widehat{\otimes}_{\C[[\hbar]]}N'_\hbar$, where
$L$ is the $-1$-eigenspace for $\theta$ in $T_\chi G\chi$, a lagrangian subspace,
and $N'_\hbar$ is the annihilator of the Lie algebra action of $L$ on $M_\hbar^{\wedge_\chi}$.
Taking the locally finite elements for the action of $\C^\times$ in $N'_\hbar$ we get
an object in $\wHC^{gr}(\Walg_\hbar)$ (where the grading on $\Walg_\hbar$
is rescaled $d/2$ times) to be denoted by $(M_\hbar)_{\dagger,\chi}$,
giving us the required functor.

Note that if $M_\hbar$ is flat over $\C[\hbar]$, then so is $N_\hbar$, as all operations above
preserve the flatness. Now we can take an object $M\in \HC(\U)^{K,\kappa}$. Equip it with a good filtration,
and take the Rees module $M_\hbar$. Rescale the grading $d$ times. We get an object
$M_{\hbar,\dagger,\chi}\in \wHC^{gr}(\Walg_\hbar)$. Consider the quotient
$M_{\hbar,\dagger,\chi}/(\hbar-1)M_{\hbar,\dagger,\chi}$. Since $\hbar$ has degree $d$,
we get a $\Z/d\Z$-grading on the quotient. Rescale the $\Z/2\Z$-grading on $\Walg$ $d/2$ times
so that only the components of degrees $0$ and $d/2$. Denote the category of $\Z/d\Z$-graded
modules over $\Walg$ by $\Walg\operatorname{-mod}^{\Z/d\Z}$. Similarly to \cite[Section 3.4]{HC},
$\wHC^{gr}(\U_\hbar)^{K,\kappa}\rightarrow \wHC^{gr}(\Walg_\hbar)$ descends to
a functor $\HC(\U)^{K,\kappa}\rightarrow\Walg\operatorname{-mod}^{\Z/d\Z}$.


Now we summarize some properties of the restriction functor. First, note that
all intermediate functors in the construction of
$$\bullet_{\dagger,\chi}: \wHC^{gr}(\U_\hbar)^{K,\kappa}\rightarrow \wHC^{gr}(\Walg_\hbar).$$
(the completion, the pushforward under an isomorphism and taking the locally finite part for
the $\C^\times$-action) are exact. So $\bullet_{\dagger,\chi}$ is exact.

Further, take $M_\hbar\in \wHC(\U_\hbar)$ and write $\Supp(M_\hbar)$
for its support in $(\g/\kf)^*$ (so that for $M\in \HC(\U,\kf)$, we have
$\Supp(R_\hbar(M))=\V(M)$). Similarly, we can talk about the supports of
modules over the completed algebras involved in the construction of
$\bullet_{\dagger,\chi}$. The support of the $\Weyl_\hbar^{\wedge_0}$-module
$\C[[L,\hbar]]$ is $(\g/\kf)^*\cap T_\chi \Orb$. It follows that
\begin{equation}\label{eq:support_restriction}
\Supp((M_\hbar)_{\dagger,\chi})=\Supp(M_\hbar)\cap S.
\end{equation}

Next, we will need the compatibility between $\bullet_{\dagger,\chi}$ and Hom's from
HC bimodules. Let $\B_\hbar$ denote a weakly HC $\U_\hbar$-bimodule and $M_\hbar$ a weakly $(K,\kappa)$-equivariant HC $\U_\hbar$-module. It is
easy to see that $\Hom_{\U_\hbar}(\B_\hbar,M_\hbar)$ is a $(K,\kappa)$-equivariant weakly HC $\U_\hbar$-module.
The following claim follows from \cite[Remark 5.7]{LY}.

\begin{Lem}\label{Lem:HC_hom_restr_compat}
There is a bi-functorial isomorphism
$$\Hom_{\U_\hbar}(\B_\hbar,M_\hbar)_{\dagger,\chi}\xrightarrow{\sim} \Hom_{\Walg_\hbar}((\B_\hbar)_{\dagger,\chi},(M_\hbar)_{\dagger,\chi}).$$
\end{Lem}

\subsection{Generalization to Dixmier algebras}\label{SS_Dixmier_generalization}
 Recall that by a Dixmier algebra one means an associative algebra $\A$ equipped with a rational action of $G$ that comes with a quantum
comoment map $\U\rightarrow \A$ (a $G$-equivariant map such that the resulting adjoint action of $\g$
on $\A$ coincides with the differential of the $G$-action) such that $\A$ is finitely generated
as a left module over $U(\g)$. So $\A$ becomes a HC bimodule over $\U$. One can show,
\cite[Lemma 5.1]{LY}, that there is a $G$-stable algebra filtration on $\A$ that is also a good bimodule
filtration. Consider the Rees algebra $\A_\hbar$ of $\A$ with respect to such a filtration
and its completion $\A_\hbar^{\wedge_\chi}$. Note that there is a natural homomorphism
$\U_\hbar\rightarrow \A_\hbar$ and hence $\U_\hbar^{\wedge_\chi}\rightarrow \A_\hbar^{\wedge_\chi}$.
Then $\A_{\hbar}^{\wedge_\chi}$
decomposes as $\Weyl_\hbar^{\wedge_0}\widehat{\otimes}_{\C[[\hbar]]}\A_\hbar'$,
where $\A_\hbar'$ is the centralizer of the image of $V$. Let $\A_{\hbar,\dagger}$
denote  the $\C^\times$-locally finite part of $\A_\hbar'$, and $\A_{\dagger}$ be its
specialization at
at $\hbar=1$. This is a filtered algebra equipped with a parity involution
$\varsigma$ and a filtered algebra homomorphism $\Walg\rightarrow \A_{\dagger}$
compatible with parity involutions.

By a weakly HC $\A_\hbar$-module we mean an $\A_\hbar$-module that becomes weakly HC
after pulling back to $\U_\hbar$. Similarly, we can talk about graded weakly HC modules.
The corresponding category will be denoted by $\wHC^{gr}(\A_\hbar)$. And similarly,
we can talk about HC $\A$-modules, weakly HC $\A_{\hbar,\dagger}$-modules, etc.
Then we get the functor $\bullet_{\dagger,\chi}:\wHC^{gr}(\A_\hbar)\rightarrow \wHC^{gr}(\A_{\hbar,\dagger})$ lifting the functor from Section \ref{SS_res_fun_constr}.

We finish this section with a discussion of an interesting class of Dixmier algebras.
Let $\tilde{\Orb}$ be a $G$-equivariant cover of a nilpotent orbit in $\g^*$,
let $\mu$ denote the corresponding map $\tilde{\Orb}\rightarrow \g^*$.
The algebra $\C[\tilde{\Orb}]$ is Poisson and graded. In many cases (for example,
when $\tilde{\Orb}\hookrightarrow \g^*$) one can assume that the degree of the Poisson
bracket is $-1$, while in general one has degree $-2$.
One can talk about filtered quantizations of $\C[\tilde{\Orb}]$. The algebra $\C[\tilde{\Orb}]$
is finitely generated, the corresponding variety $X$ is singular symplectic. Using this, one can give a classification of quantizations of $\C[\tilde{\Orb}]$, \cite[Theorem 3.4]{orbit}. These quantizations
are Dixmier algebras, \cite[Section 5.2]{orbit}, and in the case when the degree of the bracket is $-1$ (which is what we are going to assume for now), the corresponding filtration on a quantization $\A$ is a good algebra filtration. The algebra $\A_{\dagger}$ is isomorphic to
$\C[\mu^{-1}(\chi)]$ and the filtration is trivial meaning that $
\A_{\dagger,\leqslant 0}=\{0\},\A_{\dagger,\leqslant 0}=\A_{\dagger}$,
\cite[Section 5.2]{orbit}.

\begin{Rem}\label{Rem:doubled_grading}
Now consider the case when the degree of the bracket is $-2$. In this case the default filtration on
$\A$ is not compatible with the PBW filtration on $\U$, however, it is compatible with
doubled PBW filtration on $\U$. And we have the parity involution
on a quantization $\A$ and can talk about $\Z/2\Z$-graded HC modules. Their category will be denoted
by $\HC(\A)^{K,\kappa,\Z/2\Z}$. The constructions of the algebra $\A_\hbar$ and the functor $\bullet_{\dagger,\chi}$
generalize to the present setting by replacing the Rees construction with its modification.
Namely, we consider the modified Rees construction
for $\Z/d\Z$-graded spaces, compare to \cite[Section 2.3]{LY}. Let $V=\bigoplus_{a\in \Z/d\Z}V_a$ be such a space. Choose a filtration $V=\bigcup V_{\leqslant i}$ that is {\it compatible}
with the $\Z/d\Z$-grading in the following sense: each $V_{\leqslant i}$ is graded, and 
the resulting $\Z/d\Z$-grading on $\gr V$ comes from the $\Z$-grading meaning that 
$(\gr V)_i\subset \gr(V_{i\operatorname{mod} d})$.

By the {\it modified Rees module} we mean $R^{\Z/d\Z}_\hbar(V):=\bigoplus_{i} (V_{\leqslant i}\cap V_{i\mod d})\hbar^{i/d}\subset V[\hbar^{\pm 1/d}]$, this is a graded $\C[\hbar]$-submodule with
$\hbar$ of degree $d$. 

Now we can form $\A_\hbar:=R^{\Z/2\Z}_\hbar(\A)$. With this, the construction in Section
\ref{SS_res_fun_constr} goes through verbatim. 
\end{Rem}

\section{Canonical filtration}\label{S_canon_filtr}
Recall that $\Orb\subset \g^*$ is a nilpotent orbit. Let $\A$ be a Dixmier algebra that we equip
with a good filtration. The goal of this section is to construct a distinguished good filtration on $M\in \HC_\Orb^\partial(\A)^{K,\kappa}$.
For $\A=\U$ (with the PBW filtration) $K\subset G$ and $\kappa=0$ we recover the situation considered in the introduction. 

\subsection{Comparable lattices}
In what follows we will use the correspondence between good filtrations on $M$ (here we do not require them to be $K$-stable)
and graded lattices in
$M[\hbar^{\pm 1}]$. Namely, by a {\it lattice} in $M[\hbar^{\pm 1}]$ we mean a finitely generated
$\A_\hbar$-submodule $M_\hbar$ such that $M_\hbar[\hbar^{-1}]\xrightarrow{\sim} M[\hbar^{\pm 1}]$. There is a one-to-one correspondence between graded lattices in $M[\hbar^{\pm 1}]$ and good $\A$-module filtrations on $M$: we send a good filtration
to the corresponding Rees module, and we send a lattice $M_\hbar$ to the filtered module
$M_\hbar/(\hbar-1)$ (naturally identified with $M$). The condition that a good filtration is
$K$-stable translates to the condition that the lattice is $K$-stable (where $\kf$-action comes from the action of the Lie algebra
$\mathcal{J}_\hbar:=\U_\hbar\kf+ \U_\hbar\hbar\subset \U_\hbar$). A lattice is $K$-stable if and only
if it is $\mathcal{J}_\hbar$-stable. Another easy remark: if $M^1_\hbar,M^2_\hbar$ are two lattices
in $M[\hbar^{\pm 1}]$, then so are their sum and intersection. In what follows all lattices in $M[\hbar^{\pm 1}]$ (for $M\in \HC(\A)$) we consider are going to be graded $K$-stable
lattices.

Similarly, there is a one-to-one correspondence between compatible (in the sense of Remark 
\ref{Rem:doubled_grading}) good filtrations on objects
$N\in \Walg\operatorname{-mod}^{\Z/d\Z}$ and graded $\Walg_\hbar$-lattices in $N[\hbar^{\pm 1}]$.
To get from a filtration to a lattice, we take the modified Rees module $R_\hbar^{\Z/d\Z}(N)$, compare to
Remark \ref{Rem:doubled_grading}.

\begin{defi}\label{defi:comparable}
Let $M\in \HC(\A)$, and $M^1_\hbar,M^2_\hbar$ be two (graded $K$-stable) lattices
in $M[\hbar^{\pm 1}]$. We say that $M^1_\hbar,M^2_\hbar$ are {\it comparable} if
the dimension of the support of $M^i_\hbar/(M^1_\hbar\cap M^2_\hbar)$ for both $i=1,2$
is strictly less than $\dim \V(M)$.
\end{defi}

The following lemma describes basic properties of comparable lattices. The proof is easy and is left as an exercise.

\begin{Lem}\label{Lem:compar_properties}
The following claims are true:
\begin{enumerate}
\item Let $M_\hbar, M'_\hbar$ be lattices such that $M'_\hbar$ is
contained a lattice comparable with $M_\hbar$. Then $M_\hbar+ M'_\hbar$
is comparable with $M_\hbar$.
\item Let $M_\hbar, M'_\hbar$ be lattices such that $M'_\hbar$
contains a lattice comparable with $M_\hbar$. Then $M_\hbar\cap M'_\hbar$
is comparable with $M_\hbar$.
\item The comparability is an equivalence relation.
\end{enumerate}
\end{Lem}

Here is our main result about comparable lattices.

\begin{Thm}\label{Thm:max_filtration}
Suppose that $M\in \HC^\partial_\Orb(\A)^{K,\kappa}$.
In each comparability class, there is the unique \underline{maximal} (w.r.t inclusion) lattice.
Moreover, for a graded $K$-stable lattice $M_\hbar\subset M[\hbar^{\pm 1}]$, the following two conditions are equivalent:
\begin{enumerate}
\item $M_\hbar$ is maximal in its comparability class.
\item The dimension of the support of every nonzero $\gr\A$-submodule in $\gr M$ is equal to $\dim \V(M)$.
\end{enumerate}
\end{Thm}

The filtration mentioned in
Theorem \ref{Thm:main} is the maximal filtration in a certain comparability class.
Theorem  \ref{Thm:max_filtration} will be proved in
Section \ref{SS_thm_normalized_proof} after some preparation.

\begin{Rem}\label{Rem:no_filtr}
We remark that if $M\in \HC_{\overline{\Orb}}(\A)$ has a submodule supported on $\partial \Orb$, then it
cannot have a filtration satisfying (2).
\end{Rem}

\subsection{Classification}
In this section, we give a classification of lattices in $M[\hbar^{\pm 1}]$ up to comparability.
We start with some terminology.

Let $V$ be a finite dimensional module in $\A_{\dagger}\operatorname{-mod}^{\pm 1}$. An $\A_{\dagger}$-module
filtration on $V$ is called Harish-Chandra if for all $a\in \Walg^{-\theta}_{\leqslant i}$ (the superscript means the $-1$ eigenspace for $\theta$) and $v\in V_{\leqslant j}$, we have
$av\in V_{\leqslant i+j-2}$. Equivalently, the corresponding modified Rees module $R_\hbar(V)\subset V[\hbar^{\pm 1}]$ is stable under the Lie algebra action of the ideal $\mathcal{J}_\hbar\subset\Walg_\hbar$. We say that graded $\A_\hbar$-lattices in $V[\hbar^{\pm 1}]$
stable under the $\mathcal{J}_\hbar$-action are {\it HC lattices}.

Let $\Orb^1_K,\ldots, \Orb^\ell_K$ be all open $K$-orbits in $\V(M)$ (they all have dimension $\dim \V(M)=\frac{1}{2}\dim \Orb$).
Pick $\chi^i\in \Orb^i_K$. In particular, by (\ref{eq:support_restriction}),
we have $\dim M_{\dagger,\chi^i}<\infty$ for all $i$. By the construction in Section \ref{SS_res_fun_constr},
a choice of a good filtration on
$M$ gives rise to a choice of a HC filtration on $M_{\dagger,\chi^i}$
(compatible with the $\Z/d\Z$-grading).

\begin{Prop}\label{Prop:lattices_classification}
The map that sends a lattice in $M[\hbar^{\pm 1}]$ to the collection of
induced graded HC lattices in $M_{\dagger,\chi^i}[\hbar^{\pm 1}]$ defines a bijection between
the sets of
\begin{itemize}
\item[(a)] The comparability classes of  graded $K$-stable $\A_\hbar$-lattices in $M[\hbar^{\pm 1}]$.
\item[(b)] The collections of graded HC $\A_{\hbar,\dagger}$-lattices in $M_{\dagger,\chi^i}[\hbar^{\pm 1}]$ for $i=1,\ldots,k$.
\end{itemize}
\end{Prop}
\begin{proof}
First, we prove that the map is constant on comparability classes.
The functor $\bullet_{\dagger,\chi}$ kills the weakly
HC modules whose support does not intersect $\Orb_K$. It follows that if $M^1_\hbar,M^2_\hbar$
are compatible, then $(M^j_\hbar/(M^1_\hbar\cap M^2_\hbar))_{\dagger,\chi^i}=0, j=1,2, i=1,\ldots,k$. Since $\bullet_{\dagger,\chi^i}$ is an exact functor, we see that
$(M^1_\hbar)_{\dagger,\chi^i}=(M^2_\hbar)_{\dagger,\chi^i}$ for all $i$. Hence, we indeed get a
map from (a) to (b).

Now we show that the map is injective. Suppose $M^1_\hbar,M^2_\hbar$ are two lattices such that $(M^1_\hbar)_{\dagger,\chi^i}=(M^2_\hbar)_{\dagger,\chi^i}$, an equality of lattices in
$M_{\dagger,\chi^i}[\hbar^{\pm 1}]$ for all $i=1,\ldots,k$. Since $\bullet_{\dagger,\chi^i}$
is exact, it sends $(M^1_\hbar\cap M^2_\hbar)_{\dagger,\chi^i}$ to
$(M^1_\hbar)_{\dagger,\chi^i}\cap (M^2_\hbar)_{\dagger,\chi^i}=M^j_{\hbar,\dagger,\chi^i}$
for $j=1,2$. It follows that $\bullet_{\dagger,\chi^i}$ kills
$M^j_\hbar/(M^1_\hbar\cap M^2_\hbar)$. So $M^1_\hbar, M^2_\hbar$ are comparable.

Now we prove the surjectivity.
Choose graded HC $\A_{\hbar,\dagger}$-lattices $\underline{M}'^i_\hbar\subset M_{\dagger,\chi^i}[\hbar^{\pm 1}]$. We need
to prove that there is a graded $K$-stable $\A_\hbar$-lattice $M'_\hbar\subset M[\hbar^{\pm 1}]$
with $M'_{\hbar,\dagger,\chi^i}=\underline{M}'^i_\hbar$. Pick some lattice $M_\hbar
\subset M[\hbar^{\pm 1}]$. And set $\underline{M}^i_\hbar:=M_{\hbar,\dagger,\chi^i}$.

Let $\varphi_i$
denote the natural map $M_\hbar[\hbar^{-1}]\rightarrow M_\hbar^{\wedge_{\chi_i}}[\hbar^{-1}]$.
Let $L_i:=T_{\chi^i}\Orb^i_K$, so that $M_\hbar^{\wedge_{\chi_i}}$ decomposes as
$\C[[L_i,\hbar]]\otimes_{\C[[\hbar]]}\underline{M}^{i,\wedge_0}_{\hbar}$, where $\underline{M}^{i,\wedge_0}_\hbar$
stands for the $\hbar$-adic completion of $\underline{M}^i_\hbar$.
Set $M_\hbar'^{i}:=\C[[L_i,\hbar]]\otimes_{\C[[\hbar]]}\underline{M}^{i,\wedge_0}_\hbar$.
This is an $\A_\hbar^{\wedge_{\chi^i}}$-lattice
in $M_\hbar^{\wedge_{\chi_i}}[\hbar^{\pm 1}]$. It is $\C^\times$-stable by the construction.

We claim that it is
HC in the sense to be explained now. Let
$\mathcal{J}^U_\hbar, \mathcal{J}^\Weyl_\hbar, \mathcal{J}^{\Walg}_\hbar$
denote the analogs of $\J_\hbar\subset \U_\hbar$ in $\U_\hbar^{\wedge_{\chi^i}}, \Weyl_\hbar^{\wedge_0},
\Walg_\hbar^{\wedge_0}$. Then
\begin{equation}\label{eq:ideal_decomp}
\J_\hbar^\U=\Weyl_\hbar^{\wedge_0}\widehat{\otimes}_{\C[[\hbar]]}\mathcal{J}^{\Walg}_\hbar+
\mathcal{J}^\Weyl_\hbar\widehat{\otimes}_{\C[[\hbar]]}\Walg_\hbar^{\wedge_0}.
\end{equation}
When we say that a lattice in $M_\hbar^{\wedge_{\chi_i}}[\hbar^{-1}]$ is HC, we mean that it is
stable under the Lie algebra action of $\mathcal{J}^U_\hbar$. The lattice $\C[[L_i,\hbar]]$
in its localization is stable under $\mathcal{J}^\Weyl_\hbar$. Thanks to
(\ref{eq:ideal_decomp}), $M_\hbar^{\wedge_{\chi_i}}\subset M_\hbar^{\wedge_{\chi_i}}[\hbar^{-1}]$
is stable under $\J^\U_\hbar$ if and only if
$\underline{M}'^i_\hbar\subset \underline{M}'^i_\hbar[\hbar^{-1}]$
is stable under $\J^\Walg_\hbar$. The latter holds by assumption, establishing the claim in the beginning of the paragraph.

Note that there is $f>0$
such that $\hbar^{f}M_\hbar^{\wedge_{\chi_i}}\subset M_\hbar'^i\subset \hbar^{-f}M_\hbar^{\wedge_{\chi_i}}$ for all $i$ (as both $M_\hbar^{\wedge_{\chi_i}},
M_\hbar'^i$ are lattices). Set
\begin{equation}\label{eq:interm_lattice}
\tilde{M}_\hbar^i:=\hbar^{-f}M_\hbar\cap \varphi^{-1}_i(M_\hbar'^i).
\end{equation}
This is a $\mathcal{J}_\hbar$- (hence $\kf$-) stable and $\C^\times$-stable $\A_\hbar$-lattice in $M_\hbar[\hbar^{-1}]$.
Note that, since the completion functor intertwines the intersections, we have
\begin{equation}\label{eq:interm_lattice4}
(\tilde{M}_\hbar^i)^{\wedge_{\chi_i}}=M'^i_\hbar,
\end{equation}
Also,
\begin{equation}\label{eq:interm_lattice1}(\tilde{M}_\hbar^i)^{\wedge_{\chi_j}}\subset \hbar^{-e}M'^j_\hbar
\end{equation}
for some fixed
integer $e$ and all $j\neq i$.

Now we construct a lattice $M'_\hbar\subset M_\hbar[\hbar^{-1}]$ from the lattices
$\tilde{M}_\hbar^i$. Let $J_i$ denote the ideal of all elements in $S(\g)=\U_\hbar/(\hbar)$ that
vanish on $\bigcup_{j\neq i}\Orb_K^j$.
Note that $\hat{J}_i$ is $K$- and $\C^\times$-stable two-sided ideal. Also note
\begin{equation}\label{eq:interm_lattice3}
\hat{J}_iM_\hbar'^i=M_\hbar'^i
\end{equation}
because we can find an element of $J_i$
equal to $1$ at $\chi^i$. On the other hand, there is $r>0$ such that
\begin{equation}\label{eq:interm_lattice2}
\hat{J}_i^r M_\hbar'^j\subset \hbar M_\hbar'^j
\end{equation}
for all $j\neq i$. Let $\hat{I}_i$ denote the left ideal in $\A_\hbar$ generated
by $\hat{J}_i^{er}$, where $e$ is as in (\ref{eq:interm_lattice1}),
$r$ is as in (\ref{eq:interm_lattice2}). (\ref{eq:interm_lattice3}) gives
\begin{equation}\label{eq:interm_lattice3a}
\hat{I}_iM_\hbar'^i=M_\hbar'^i,
\end{equation}
while from (\ref{eq:interm_lattice2}) we get
\begin{equation}\label{eq:interm_lattice2a}
\hat{I}_i M_\hbar'^j\subset \hbar^e M_\hbar'^j
\end{equation}

We set
$$M'_\hbar:=\hbar^{f} M_\hbar+\sum_{i=1}^\ell \hat{I}_i\tilde{M}_\hbar^i,$$
where $f$ is as in (\ref{eq:interm_lattice}). This $\A_\hbar$-submodule is $K$-stable and
graded by the construction. It generates $M_\hbar[\hbar^{-1}]$ as an
$\A_\hbar[\hbar^{-1}]$-module because it contains $\hbar^f M_\hbar$. It is a lattice
because $M_\hbar$ and all $\tilde{M}_\hbar^i$ are lattices. We have $\tilde{M}_\hbar^{\wedge_{\chi_i}}=M'^i_\hbar$ for all $i$
because
\begin{itemize}
\item
the completion functor is exact and so intertwines the sums,
\item $(\hat{I}_i\tilde{M}_\hbar^i)^{\wedge_{\chi_i}}=M_\hbar'^i$ by (\ref{eq:interm_lattice4}) combined with (\ref{eq:interm_lattice3a}),
\item $(\hbar^d M_\hbar)^{\wedge_{\chi^i}}\subset M_\hbar'^i$
by (\ref{eq:interm_lattice}),
\item  $(\hat{I}_j\tilde{M}_\hbar^i)^{\wedge_{\chi_i}}\subset M_\hbar'^i$  by (\ref{eq:interm_lattice2a}) combined with (\ref{eq:interm_lattice1}).
\end{itemize}
So we have constructed a lattice $M'_\hbar\subset M_\hbar[\hbar^{-1}]$ with required properties.
This finishes the proof of the surjectivity, and hence of the proposition.
\end{proof}

\subsection{Fd-comparable filtrations}
In this section we consider the following situation. Let $\Orb_K$ be a
nilpotent $K$-orbit in $(\g/\kf)^*$. Pick $\chi\in \Orb_K$.
Let $S$ denote a $\theta$-stable Slodowy slice at $\chi$. Consider the algebra $\A_{\dagger}$. Take $N\in \A_\dagger\operatorname{-mod}^{\Z/d\Z}$. Suppose that $N$ does not have
nonzero finite dimensional $\A_\dagger$-submodules.

\begin{defi}
Two  graded  $\A_{\hbar,\dagger}$-lattices $N_\hbar^{1}$ and $N_\hbar^{2}$ are said to be {\it fd-comparable} if $N_\hbar^1/(N_\hbar^1\cap N_\hbar^2), N_\hbar^2/(N_\hbar^1\cap N_\hbar^2)$
are finite dimensional.
\end{defi}

Note that being fd-comparable is an equivalence relation, compare to Lemma \ref{Lem:compar_properties}, so we can talk about the fd-comparability classes of lattices.

The goal of this section is to prove the following result that will be used in the proof of
Theorem \ref{Thm:max_filtration}.

\begin{Prop}\label{Prop:max_filtr_dim1}
Recall that we assume that $N$ does not have finite dimensional submodules.
In each fd-comparability class, there is the unique maximal (w.r.t. inclusion) element.
\end{Prop}
\begin{proof}
Pick a graded $\A_{\hbar,\dagger}$-lattice $N_\hbar$. For each $k\geqslant 0$, we have the unique lattice
fd-comparable with $N_\hbar$ and maximal among all fd-comparable lattices contained in
$\hbar^{-k}N_\hbar$: the sum of all fd-comparable lattices, which coincides with the sum of finitely many of them because $\A_{\hbar,\dagger}$ is a Noetherian algebra (it is a finitely generated module over the
Noetherian algebra $\Walg_\hbar$). Denote this maximal lattice by $N_\hbar^k$.
So we have the increasing chain of fd-comparable lattices
$$N_\hbar=N^0_\hbar\subset N^1_\hbar\subset\ldots\subset N^k_\hbar\subset\ldots$$
and the claim of the proposition amounts to showing that this chain terminates.
The proof of this is in several steps.

{\it Step 1}. We claim that for each $i$ we have
\begin{align}\label{eq:containment1}
&\hbar N_\hbar^{i+1}\subset N_\hbar^i,\\\label{eq:containment2}
&N_\hbar^i=N_\hbar^{i+1}\cap \hbar^{-i}N_{\hbar}.
\end{align}
To show (\ref{eq:containment1}) observe that $\hbar N_{\hbar}^{i+1}+N_\hbar$ is
fd-comparable to $N_\hbar$ by the direct analog of (1) of Lemma \ref{Lem:compar_properties} and is  contained in $\hbar^{-i}N_\hbar$. (\ref{eq:containment2})
is similar: we use the direct analog of (2) of Lemma \ref{Lem:compar_properties}.

{\it Step 2}. We claim that the inclusions $\hbar N_{\hbar}^{i+1}\subset N_\hbar^i, \hbar N_{\hbar}^{i}\subset N_\hbar^{i-1}$ induce an embedding
\begin{equation}\label{eq:containment3}
\hbar N_\hbar^{i+1}/\hbar N_{\hbar}^i\hookrightarrow N_{\hbar}^i/N_{\hbar}^{i-1}, \forall i\geqslant 1.
\end{equation}
Indeed, thanks to (\ref{eq:containment1}), we have
$$N_\hbar^{i+1}/N_\hbar^i=N_\hbar^{i+1}/(N_\hbar^{i+1}\cap \hbar^{-i}N_{\hbar})
\hookrightarrow (N_{\hbar}^{i+1}+\hbar^{-i}N_\hbar)/(\hbar^{-i}N_\hbar).$$
The multiplication by $\hbar$ identifies the latter space with
$(\hbar N_{\hbar}^{i+1}+\hbar^{1-i}N_\hbar)/(\hbar^{1-i}N_\hbar)$. Thanks
to (\ref{eq:containment1}), this space includes into
$$(N_{\hbar}^{i}+\hbar^{1-i}N_\hbar)/(\hbar^{1-i}N_\hbar)=N_\hbar^{i}/N_{\hbar}^{i-1},$$
where the equality is thanks to (\ref{eq:containment2}) (with $i$ replaced by $i-1$).

{\it Step 3}. Note that all spaces $N_\hbar^{i+1}/N_{\hbar}^i$ are finite dimensional
by the definition of fd-comparable lattices. Thanks to (\ref{eq:containment3}), the sequence
$N_{\hbar}^{i+1}/N_{\hbar}^i$ stabilizes. We need to show that it stabilizes to $0$.

Assume the contrary.
Note that all spaces $N_{\hbar}^{i+1}/N_\hbar^i$ are graded, and (\ref{eq:containment3})
is homogeneous.
Consider the sequence $\hbar^{i}N_\hbar^{i}\subset N_\hbar$, it is decreasing by  (\ref{eq:containment1}). Let $K_\hbar$ denote the intersection $\bigcap_{i\geqslant 0}\hbar^{i}N_\hbar^{i}$.

We claim that
\begin{equation}\label{eq:containment4}
K_\hbar/\hbar K_\hbar\xrightarrow{\sim} \varprojlim N_\hbar^{i+1}/N_{\hbar}^i.
\end{equation}
Indeed, take $j$ such that the maps
(\ref{eq:containment3}) are isomorphisms for $i>j$. This is equivalent to
$N^{i}_\hbar=N^{i-1}_\hbar+\hbar N^{i+1}_\hbar$ for all $i>j$,
equivalently
\begin{equation}\label{eq:containment5}
\hbar^i N^i_\hbar=\hbar^{i-j}(\hbar^jN^{j}_\hbar)+ \hbar^{i+1}N^{i+1}_\hbar, \forall i>j.
\end{equation}
 Now take
a homogeneous element $n\in N_\hbar^{j+1}/N_{\hbar}^j$ of degree $f$, and let
$f'$ be the minimal degree in $N^j_\hbar$ (that exists because $N^j_\hbar$ is a finitely
generated module over the positively graded algebra $\Walg_\hbar$). Take the component of degree $i$
in (\ref{eq:containment5}) for $i-j>f-f'$, and get $(\hbar^i N^i_\hbar)_{f}=
(\hbar^{i+1}N^{i+1}_\hbar)_{f}$. This implies that (\ref{eq:containment4})
is an isomorphism on the elements of degree $f$. So, it is an isomorphism.

{\it Step 4}.
In particular, $K_\hbar$ is a nonzero graded $\A_{\hbar,\dagger}$-submodule in $N_\hbar$ that is supported at the
point $\chi\in S\cap (\g/\kf)^*$. Then $K_\hbar/(\hbar-1)K_\hbar$ is a nonzero finite dimensional
$\A_{\dagger}$-submodule of $N$. This contradiction shows that (\ref{eq:containment3}) stabilizes
at $0$ and finishes the proof.
\end{proof}

\subsection{Proof of Theorem \ref{Thm:max_filtration}}\label{SS_thm_normalized_proof}
We start the following result that will be used in the proof.

\begin{Lem}\label{Lem:no_fin_dim_subs}
Let $M\in \HC_{\Orb}^\partial(\A)^{K,\kappa}$. Pick a point $\chi\in \V(M)\setminus \Orb$. Then $M_{\dagger,\chi}$ has no nonzero finite dimensional $\A_{\dagger,\chi}$-submodules.
\end{Lem}
\begin{proof}
It is enough to prove this claim for $\A=\U$.

Assume the contrary, let $N$ be a nonzero finite dimensional $\Walg$-submodule of $M_{\dagger,\chi}$.
Let $\Orb':=G\chi$. We can assume
that $N$ is irreducible. Using \cite[Theorem 1.2.2]{HC}, we see that there is an ideal
$\J\subset \U$ with $\V(\U/\J)=\overline{\Orb}'$ such that $\J_{\dagger}$ annihilates $N$.
It follows that $\Hom_{\Walg}((\U/\J)_{\dagger,\chi}, M_{\dagger,\chi})\neq \{0\}$.
By Lemma \ref{Lem:HC_hom_restr_compat} (applied to HC bimodules/modules rather than weakly
HC ones), we have $\Hom_{\U}(\U/\J,M)\neq 0$, equivalently,
the annihilator of $\J$ in $M$ is nonzero. This annihilator must be supported on $\overline{\Orb}'\cap (\g/\kf)^*$, which is impossible because $M\in \HC^\partial_{\Orb}(\g,K)$.
\end{proof}

\begin{proof}[Proof of Theorem \ref{Thm:max_filtration}]
First, we prove that a maximal graded $K$-stable lattice in the comparability class of $M_\hbar$ exists
(the uniqueness is clear). As in the proof of Proposition \ref{Prop:max_filtr_dim1}, we have the maximal such lattice contained in $\hbar^{-i}M_\hbar$, denote it by $M_\hbar^i$. We need to show that the sequence $(M_\hbar^i)$ stabilizes.

Let $\chi\in \V(M)\setminus \Orb$. We prove that $(M_\hbar^i)_{\dagger,\chi}$ terminates by induction on
$\dim \V(M)-\dim K\chi$. The base case is when $\dim \V(M)-\dim K\chi=1$. Here the lattices $M^i_{\hbar,\dagger,\chi}\subset M_{\dagger,\chi}[\hbar^{\pm 1}]$ are fd-comparable. Thanks to Lemma
\ref{Lem:no_fin_dim_subs}, $M_{\dagger,\chi}$ satisfies the assumptions of Proposition \ref{Prop:max_filtr_dim1} and we are done by this proposition.

Now suppose that $M^i_{\hbar,\dagger,\chi}$ stabilizes for all $\chi$ with
\begin{equation}\label{eq:dim_induction}\dim \V(M)-\dim K\chi<\ell\end{equation}
(with some $\ell$). As we only have finitely many orbits $K\chi$, we can assume that there is $j$ with
$M^i_{\hbar,\dagger,\chi}=M^j_{\hbar,\dagger,\chi}$ for all $i>j$ and all
$\chi$ satisfying (\ref{eq:dim_induction}). It follows that the support of $M^i_{\hbar}/M^j_\hbar$
has dimension at most $\dim \V(M)-\ell$. So, for $\chi$ with $\dim \V(M)-\dim K\chi=\ell$, the
lattices $M^i_{\hbar,\dagger,\chi}\subset M_{\dagger,\chi}[\hbar^{\pm 1}]$ with $i\geqslant j$
are fd-comparable. Similarly to the previous paragraph, this sequence stabilizes. This
establishes the induction step and finishes the proof of the existence of a maximal lattice
in a given comparability class.

Now we prove the equivalence of (1) and (2). First, suppose that $M_\hbar$ is maximal in its comparability class. Let $N\subset \gr M$ be a $\C[\g^*]$-submodule such that $\dim \operatorname{Supp}(N)<\dim \V(M)$.
We can choose it to be maximal of all submodules with this property. Then it must be $\C^\times$-stable
as the action of $\C^\times$ maps one submodule with the required property into another.
It also must be $K$-stable for a similar reason. The preimage of $N$ in $\hbar^{-1}M_\hbar$ under the projection
to $\hbar^{-1}M_\hbar/M_\hbar\cong \gr M$ is comparable with $M_\hbar$, a contradiction.

Now suppose that $\gr M$ has no submodules with dimension of support strictly less than $\dim \V(M)$.
Suppose $M_\hbar$ is not maximal. Then we can find a comparable graded $K$-stable lattice $\tilde{M}_\hbar$ between $M_\hbar$ and $\hbar^{-1}M_\hbar$. Then $\tilde{M}_\hbar/M_\hbar$
is a submodule in $\hbar^{-1}M_\hbar/M_{\hbar}\cong \gr M$ whose support has dimension strictly
less than $\dim \V(M)$. A contradiction.
\end{proof}

\subsection{A question of Vogan}
The goal of this section is to use Theorem \ref{Thm:max_filtration} to prove Theorem
\ref{Thm:codim1_connected} from Introduction.

\begin{proof}[Proof of Theorem \ref{Thm:codim1_connected}]
Assume the contrary, $X:=\V(M)\setminus \V(M)^{\geqslant 2}$ splits as a disjoint union of two closed nonempty subvarieties $X_1$ and $X_2$.

Thanks to \cite[Lemma 2.5]{LY}, we can find
\begin{itemize}
\item
a connected reductive group $K$ with a homomorphism to $G$
that identifies the Lie algebra of $K$ with $\kf$,
\item and a character $\kappa$ of $\kf$
\end{itemize}
so that $M$ becomes a strongly $(K,\kappa)$-equivariant $\U$-module.

Equip $M$ with a $K$-stable good filtration that is maximal in its comparability class. Let $I$ denote
the annihilator of $\gr M$ in $S(\g)$ and $A:=S(\g)$. By condition (2) of Theorem \ref{Thm:max_filtration}, a non-zero divisor of $A$ is also a non-zero divisor in $\gr M$.
Let $(\gr M)^0$ denote the restriction of $\gr M$ to $X_1$: formally it is
the localization of $A$-module $\gr M$ to the open subset $\operatorname{Spec}(A)\setminus [\V(M)^{\geqslant 2}\sqcup X_2]$. The key observation is that $\Gamma((\gr M)^0)$ is finitely generated.
Indeed, the locus in $\operatorname{Spec}(A)$, where $\gr M$ fails to be Cohen-Macaulay
has codimension greater than $1$. This locus has to be $K$-stable because of the $K$-action
on $\gr M$. So it is contained in $\V(M)^{\geqslant 2}$. Hence $(\gr M)^0$ is Cohen-Macaulay.
The push-forward of a maximal Cohen-Macaulay coherent sheaf from an open subset whose complement has codimension at least $2$ is coherent. Moreover, the support of $\Gamma((\gr M)^0)$ is contained in $\overline{X}_1$.

Then the argument proceeds as in the proof of \cite[Theorem 4.6]{Vogan}. We can microlocalize $M$
to a sheaf on $\V(M)$, in particular, consider the restriction $M^0$ of this sheaf to $X_1$. This sheaf carries a complete and separated $K$-stable filtration whose associated graded is $(\gr M)^0$. It follows that $\gr\Gamma(M^0)\hookrightarrow \Gamma((\gr M)^0)$, hence is
a finitely generated $S(\g^*)$-module. So $\Gamma(M^0)$ carries a good filtration, hence is a HC $(\g,K,\kappa)$-module. By the construction, its associated variety is contained in the support
of  $\Gamma((\gr M)^0)$, which, in its turn, is contained in $\overline{X}_1$.

On the other hand, by the construction, we have a $\U$-module homomorphism $M\rightarrow M^0$,
which is nonzero because $X_1\subset \V(M)$. So we get a nonzero homomorphism $M\rightarrow \Gamma(M^0)$.
Since $M$ is simple, this homomorphism has to be injective. But this leads to a contradiction:
as we have seen above in this proof, $\V(M)\not\subset X_1$. We have proved that $\V(M)\setminus \V(M)^{\geqslant 2}$ is connected.
\end{proof}

\subsection{Normalized and canonical filtrations}
Recall the integer $d:=2|\ker\iota|$, where $\iota$ is the homomorphism $K\rightarrow G$.

Let $\A$ be a Dixmier algebra. We assume that it is equipped with a good filtration subject to certain conditions. Namely, fix a nilpotent orbit $\Orb\subset \g^*$, take a point $\chi\in \Orb$ and let $S$ be a
Slodowy slice through $\chi$. The restriction $(\gr\A)|_S$ has a natural grading, it is induced 
by the action $t\mapsto t^{-d}\gamma(t)$.

Here are the two conditions on $\A$ that we need:
\begin{itemize}
\item[(a)] $\gr\A$ is commutative,
\item[(b)] $(\gr\A)|_S$ is positively graded (meaning that $(\gr\A)|_{S,i}$ is zero for $i<0$
and is 1-dimensional for $i=0$). Moreover, $(\gr\A)_{S,i}=\{0\}$ for $0<i<d$.
\end{itemize}

For example, $\A=\U$ satisfies both conditions.
Another example is when $\A$ is a quantization of $\C[\Orb]$ (or of $\C[\tilde{\Orb}]$, where
$\tilde{\Orb}$ is an equivariant cover of $\Orb$ such that the scaling action of $\C^\times$
on $\Orb$ lifts to $\tilde{\Orb}$). As was noted in Section \ref{SS_Dixmier_generalization},
$(\gr\A)_S$ is concentrated in degree $0$.



\begin{defi}\label{eq:trivial_filtration}
Let $N$ be a finite dimensional $\Z/d\Z$-graded $\A_{\dagger}$-module. By the {\it trivial
filtration} on $N$ we mean the filtration given by $N_{\leqslant -1}=\{0\}, N_{\leqslant i}=\sum_{j=0}^i
N_{i\operatorname{mod} d}$ for $i=0,\ldots,d-1$.
\end{defi}

The following lemma describes easy properties of this filtration.

\begin{Lem}\label{Lem:triv_filtr_properties}
The following claims are true:
\begin{itemize}
\item The trivial filtration is a compatible HC filtration.
\item The associated graded of the trivial filtration is annihilated by $(\gr\A_\dagger)_{>0}$.
\item Suppose that $N$ is equipped with a compatible filtration such that the only nonzero components in $\gr N$ are in degrees $0,\ldots,d-1$. Then the filtration is trivial.
\end{itemize}
\end{Lem}
\begin{proof}
(1) and (2) easily follow from (b). (3) is an easy exercise.
\end{proof}

Now we can define the canonical filtration on an object $M\in \HC^\partial_{\Orb}(\A,K,\kappa)$.
Recall that each comparability class of good $K$-stable $\A$-module filtrations on $M$ has the unique maximal element, Theorem \ref{Thm:max_filtration}. Further, if $\Orb^1_K,\ldots,\Orb^\ell_K$ be all $K$-orbits of dimension $\frac{1}{2}\dim \Orb$ in $\V(M)$, then there is a bijection between the comparability classes of good filtrations on $M$ and the $\ell$-tuples of HC filtrations compatible with the $\Z/d\Z$-gradings on
$M_{\dagger,\chi^i}$  for $\chi^i\in \Orb^i_K$, Proposition \ref{Prop:lattices_classification}.

\begin{defi}\label{defi:canon_filtration}
By a {\it normalized filtration} on $M$ we mean a good $K$-stable $\A$-module filtration such that
the corresponding filtrations on the modules $M_{\dagger,\chi^i}$ are trivial.
By the {\it canonical filtration} on $M$ we mean the unique normalized filtration such that
the corresponding lattice on $M[\hbar^{\pm 1}]$ is maximal.
\end{defi}

Our next goal is to characterize the canonical filtration in elementary terms, i.e., without using
restriction functors, proving a direct generalization of Theorem \ref{Thm:main}.

\begin{Thm}\label{Thm:canonical_charact}
For a good $K$-stable $\A$-module filtration on $M\in \HC^\partial_{\Orb}(\A,K,\kappa)$, the following two conditions are equivalent:
\begin{enumerate}
\item The filtration is canonical.
\item
\begin{itemize}
\item[(a)] The annihilator $I$ of $\gr M$ in $\gr \A$ is radical.
\item[(b)] Every nonzero divisor in $(\gr \A)/I$
is also a nonzero divisor in $\gr M$.
\item[(c)] For each $i=1,\ldots,\ell$, the graded module $(\gr M)_{\chi^i}$ is concentrated in
degrees $0,\ldots,d-1$.
\end{itemize}
\end{enumerate}
\end{Thm}
\begin{proof}
We start with (1)$\Rightarrow$(2).

We claim that
\begin{itemize}
\item[(*)] The $S(\g/\kf)$-module
$\sqrt{I}/I$ is supported away from $\bigsqcup_{i=1}^\ell \Orb^i_K$.
\end{itemize}
We write $\chi$ for $\chi^i$ and $L$ for the lagrangian subspace $L_i$.
By Lemma \ref{Lem:triv_filtr_properties}, $\gr (M_{\dagger,\chi})$ is killed by $(\gr\A_\dagger)_{>0}$. Note that $\gr M$ is an object in
$\wHC^{gr}(\A_\hbar)$, and $(\gr M_{\dagger,\chi})\cong \gr (M_{\dagger,\chi})$.
In particular, $(\gr M)^{\wedge_{\chi}}\cong \C[[L]]\otimes \gr (M_{\dagger,\chi})$.
The annihilator of this module in $(\gr\A)^{\wedge_\chi}\cong \C[[T_{\chi}\Orb]]\widehat{\otimes}(\gr\A_\dagger)^{\wedge_\chi}$ 
is the vanishing ideal of
$\C[[T_{\chi}\Orb]]\widehat{\otimes}(\gr\A_\dagger)^{\wedge_\chi}_{>0}$. So it is radical. We note that the annihilator of the completion
is the completion of the annihilator. Also the completion of the radical is the radical of the completion. (*) follows.

Now note that modulo (*), (a) is equivalent to (2) of Theorem \ref{Thm:max_filtration}.
(b) is equivalent to (2) of Theorem \ref{Thm:max_filtration} unconditionally. And (c) is
satisfied because the filtration is normalized.

To prove (2)$\Rightarrow$(1), note that (a) and (b) together imply (2) of Theorem \ref{Thm:max_filtration}. Hence the filtration on $M$ is maximal in its comparability class.
Thanks to (c) and (3) of Lemma \ref{Lem:triv_filtr_properties}, the filtration is normalized.
So, it is canonical.
\end{proof}

To finish the section we prove the following functoriality result.

\begin{Prop}\label{Prop:functoriality}
Let $M_1,M_2\in \HC^\partial_\Orb(\A,K,\kappa)$ and $\varphi:M_1\rightarrow M_2$ be a homomorphism.
Then $\varphi$ is filtration preserving with respect to the canonical filtrations.
\end{Prop}
\begin{proof}
The proof is in two steps.

{\it Step 1}. First, assume $M_i=M, i=1,2,$ and $\varphi$ is an automorphism. Note that $\varphi$
(extended to $M[\hbar^{\pm 1}]$ in an obvious way) sends comparable lattices to comparable lattices,
preserving inclusions. So, one just needs to check that a normalized filtration is sent to a normalized
filtration. This is because any morphism in the category of finite dimensional objects in $\A_{\dagger}\operatorname{-mod}^{\Z/d\Z}$ 
is compatible with the trivial filtrations.

{\it Step 2}. We deduce the general case from Step 1. Note that $M_1\oplus M_2$ carries the direct sum filtration. It is canonical, this follows for example from Theorem \ref{Thm:canonical_charact}.
Now consider the automorphism of $M_1\oplus M_2$ sending $(m_1,m_2)$ to $(m_1,m_2+\varphi(m_1))$.
Applying the claim of Step 1 finishes the proof.
\end{proof}

\end{document}